\documentclass[11pt]{amsart}

\usepackage{amsmath,amscd,amssymb,latexsym,color}
\usepackage[makeroom]{cancel}
\usepackage{tikz-cd}
\usepackage{longtable}

\input xy \xyoption{all}

\newcommand{\sm}{\left(\smallmatrix}
\newcommand{\esm}{\endsmallmatrix\right)}
\newcommand{\mat}{\begin{pmatrix}}
\newcommand{\emat}{\end{pmatrix}}
\renewcommand{\c}{\mathfrak{c}}
\renewcommand{\o}{\omega}

\newcommand{\Q}{\mathbb Q}
\newcommand{\Z}{\mathbb Z}
\newcommand{\C}{\mathbb C}

\renewcommand{\H}{\mathbb H}
\renewcommand{\P}{\mathbb P}

\newtheorem{thm}{Theorem}
\newtheorem{lem}[thm]{Lemma}

\newtheorem{prop}[thm]{Proposition}

\newtheorem{alg}[thm]{Algorithm}

\theoremstyle{definition}

\newtheorem{rmk}[thm]{Remark}
\numberwithin{equation}{section}
\numberwithin{thm}{section}

\begin{document}

\title[Tetraelliptic modular curves $X_1(N)$]{Tetraelliptic modular curves $X_1(N)$}

\author{Daeyeol Jeon}
\address{Department of Mathematics Education, Kongju National University, Gongju, 32588 South Korea}
\email{dyjeon@kongju.ac.kr}

\thanks{This research was supported by Basic Science Research Program through the National Research Foundation of Korea (NRF) funded by the Ministry of Education (2019R1F1A1060149, 2022R1A2C1010487).}

\begin{abstract} In this paper, we determine all tetraelliptic modular curves $X_1(N)$ over $\Q$, and find some tetraelliptic maps $\phi_N$ from $X_1(N)$ to elliptic curves for those tetraelliptic $X_1(N)$. Also we will construct $\phi_N$ explicitly as rational functions. Moreover, we will show that all $\phi_N$ we found are Galois and find elliptic curves with torsion subgroup $\Z/17\Z$ over cyclic quartic number fields by using the cyclic map $\phi_{17}$.  
\end{abstract}
\maketitle

\renewcommand{\thefootnote}%
             {}
 {

\section{Introduction}\label{sec:Introduction}


Let $X$ be a smooth, projective curve over a number field $K$ of genus $g(X)\geq 2$.
Suppose that there is a map $\phi:X\to E$ of degree $d$ for some elliptic curve $E$.
Then $X$ is said to be {\it bielliptic}, {\it trielliptic}, and {\it tetraelliptic} if $d=2,3$, and $4$, respectively.

For a positive integer $N$, the modular curve $X_1(N)$ (with cusps removed) parametrizes isomorphism classes of pairs $(E,P)$, where $E$ is an elliptic curve and $P$ a torsion point of order $N$ on $E$.

The first author and Kim \cite{JK} determined all bielliptic curves $X_1(N)$.
Recently, the author \cite{J} determined all trielliptic curves $X_1(N)$ over $\Q$, i.e. they admit $\Q$-rational {\it trielliptic maps} from $X_1(N)$ to elliptic curves $E$ over $\Q$ of degree 3.
Also he construct explicit $\Q$-rational trielliptic maps from trielliptic $X_1(N)$ over $\Q$ to elliptic curves.
Interestingly, he rediscovered the Najman's elliptic curve labelled 162B1 over $\Q(\zeta_9)^+$ which is an elliptic
curve over a cubic number field with torsion $\Z/21\Z$ in constructing a trielliptic map from $X_1(21)$ to an elliptic curve.
Indeed, the elliptic curve 162B1 corresponds to a sporadic point on $X_1(21)$ of degree $3$. 
In general, a point on $X_1(N)$ is called a point of degree $n$ if it is defined over a number field of degree $n$ over $\Q$.
In particular, we call a point of degree $4$ a {\it quartic point}.

In this paper, we determine all tetraelliptic modular curves $X_1(N)$ over $\Q$, and find some tetraelliptic maps $\phi_N$ from tetraelliptic $X_1(N)$ to elliptic curves for those tetraelliptic $X_1(N)$.
Also we will construct $\phi_N$ explicitly as rational functions.  
Constructing explicit expressions of maps from an algebraic curve with a high genus to an elliptic curve seems to be very meaningful from a computational point of view. 
At least the author could not find any literature showing these explicit expressions except for the author's previous work \cite{J}.
Moreover, we could show that all $\phi_N$ we found are Galois and find elliptic curves with torsion subgroup $\Z/17\Z$ over cyclic quartic number fields by using the cyclic map $\phi_{17}$.   

Our main theorem is as follows:
\begin{thm}\label{main} $X_1(N)$ is tetraelliptic over $\Q$ if and only if $N=17,20,24,26,28,30,32$.
Moreover, they all have $\Q$-rational Galois tetraelliptic maps, i.e. they admit $\Q$-rational Galois maps of degree $4$ from $X_1(N)$ to elliptic curves.
Especially, they are cyclic for $N=17,20,26,30$ and $32$, and the others are not.
\end{thm}

We would like to make a few additional comments on the torsion subgroups of elliptic curves.
As is well known, Mazur \cite{M} determined all possible group structures of the torsion subgroups of elliptic curves over $\Q$.
Kamienny \cite{K}, Kenku and Momose \cite{KM} completed the same work as Mazur over quadratic number fields.

After these results, attention shifted toward number fields of higher degrees. 
Jeon, Kim and Schweizer \cite{JKS1} found all the torsion groups that appear infinitely often as one runs through all elliptic curves over all cubic number fields as follows:
\begin{equation}\label{JKS} \begin{array}{ll}
    \Z/N\Z, & N=1-16,18,20, \\
    \Z/2\Z\oplus \Z/2N'\Z, & N'=1-7.\\
\end{array} \end{equation}

Since Najman's elliptic curve was discovered in \cite{N}, it was formally conjectured in \cite{W} that the only possible torsion groups for elliptic curves over cubic number fields are the ones in \eqref{JKS} and $\Z/21\Z$.
Finally, Derickx, Etropolski, van Hoeij, Morrow, and Zureick-Brown \cite{DE} proved this conjecture, and hence the torsion group $E(K)_{tors}$ of an elliptic curve $E$ over a cubic number field $K$ must be isomorphic to one of the following $26$ types:
\begin{equation}\label{DEHMZ} \begin{array}{ll}
    \Z/N\Z, & N=1-16,18,20, 21\\
    \Z/2\Z\oplus \Z/2N'\Z, & N'=1-7.\\
\end{array} \end{equation}

Over quartic number fields $K$, it is proved in~\cite{JKP} that all the group structures occurring infinitely often as torsion groups
$E(K)_{\operatorname{tors}}$ are exactly the following $38$ types:
\begin{equation}\label{JKP}\begin{array}{ll}
    {\mathbb Z}/N_1{\mathbb Z}, & N_1=1-18,20,21,22,24 \\
    \Z/2\Z\oplus \Z/2N_2\Z, & N_2=1-9 \\
    \Z/3\Z\oplus \Z/3N_3\Z, & N_3=1-3 \\
    \Z/4\Z\oplus \Z/4N_4\Z, & N_4=1-2\\
    \Z/5\Z\oplus \Z/5\Z, & \\
    \Z/6\Z\oplus \Z/6\Z, &
\end{array} \end{equation}
where $K$ varies over all quartic number fields and $E$ varies over all elliptic curves over $K$.
In the case of quartic number fields yet, a complete classification of torsion subgroups has not been achieved. 
Unlike the cubic number fields, the sporic points on $X_1(N)$ of degree $4$ have not been discovered yet, 
so there is still a possibility that \eqref{JKP} will be a full classification over quartic number fields, and it is very interesting to see how this result ends.

Lastly, we will give a comment about our construction of elliptic curves with torsion subgroup $\Z/17\Z$ over cyclic quartic number fields.
Chou \cite{C} proved that if $E$ is an elliptic curve defined over $\Q$ and $K$ is a Galois quartic extension of $\Q$, then $E(K)_{\operatorname{tors}}$ is one of the following $27$ types:
\begin{equation}\label{Ch}\begin{array}{ll}
    {\mathbb Z}/N_1{\mathbb Z}, & N_1=1-16,N_1\neq 11,14 \\
    \Z/2\Z\oplus \Z/2N_2\Z, & N_2=1-6,8 \\
    \Z/3\Z\oplus \Z/3N_3\Z, & N_3=1,2 \\
    \Z/4\Z\oplus \Z/4N_4\Z, & N_4=1,2\\
    \Z/5\Z\oplus \Z/5\Z, & \\
    \Z/6\Z\oplus \Z/6\Z, &
\end{array} \end{equation}
The above result implies that our elliptic curves with torsion $\Z/17\Z$ over cyclic quartic number fields have no $\Q$-rational Weierstrass equations.


\section{Preliminaries}\label{sec:Preliminaries}

Let $\Gamma(1)={\rm SL}_2(\mathbb Z)$ be the full modular group. 
For any positive integer $N$, we have subgroups $\Gamma(N)$, $\Gamma_1(N)$ and $\Gamma_0(N)$ of $\Gamma(1)$ defined by matrices $\sm a&b\\c&d\esm$ that are congruent modulo $N$ to $\sm 1&0\\0&1\esm$, $\sm 1&*\\0&1\esm$ and $\sm *&*\\0&*\esm$, respectively. 
We let $X(N)$, $X_1(N)$ and $X_0(N)$ be the modular curves defined over $\mathbb Q$ associated with $\Gamma(N)$, $\Gamma_1(N)$ and $\Gamma_0(N)$, respectively.
There are some more modular curves $X_\Delta(N)$ associated with the subgroups $\Gamma_\Delta(N)$ of $\Gamma_0(N)$ defined by matrices $\sm a&b\\c&d\esm$ with $a\in\Delta$, where $\Delta$ is a subgroup of $(\Z/N\Z)^*$ that contains $-1$. For $\Delta=\{\pm1\}$, this is $X_1(N)$.  

Now we introduce a special form of an elliptic curve that is used frequently to describe a pair $(E,P)$ which corresponds to a non-cuspidal point of $X_1(N)$ (cf. \cite{H}).
The {\it Tate normal form} of an elliptic curve with a point $P=(0,0)$ is
$$E=E(b,c) : Y^2+(1-c)XY -bY=X^3 -bX^2,$$ and this is nonsingular if and only if $b\neq0$.
On the curve $E(b,c)$, we can use the chord-tangent method to derive the following:
\begin{align}\label{eq:nP}
P=&(0,\ 0),\\\notag
2P =&(b, \ bc),\\\notag
3P =&(c, \ b-c),\\\notag
4P =&\left(r(r-1), \ r^2(c-r+1)\right); \ \ b=cr,\\\notag
5P =& \left(rs(s-1), \ rs^2(r-s)\right); \ \ c=s(r-1), \\ \notag
6P =& \left( \frac{s(r-1)(r-s)}{(s-1)^2}, \ \frac{s^2(r-1)^2(rs-2r+1)}{(s-1)^3}  \right) \notag
\end{align}

The condition $NP= O$ in $E(b,c)$ gives a defining equation for $X_1(N)$.
For example, $11P=O$ implies $5P=-6P$, so
$$X_{5P}=X_{-6P}=X_{6P},$$
where $X_{nP}$ denotes the $X$-coordinate of the $n$-multiple $nP$ of $P$.
\eqref{eq:nP} implies that
\begin{equation}\label{mod13}
rs(s-1) = \frac{s(r-1)(r-s)}{(s-1)^2}.
\end{equation}
Without loss of generality, the cases $s=0,1$ may be excluded.
Then, \eqref{mod13} becomes
\begin{align*}
F_{11}(r,s):=&\, rs^3 - 3rs^2 - r^2 + 4rs - s=0,
\end{align*}
which is one of the equations for $X_1(11)$, called the {\it raw form} of
$X_1(11)$.  
By the coordinate changes $r =1+ xy$ and $s=1-x$, we have that
$$f_{11}(x,y):=y^2 + (x^2 + 1)y + x=0.$$
This solves the moduli problem of $X_1(11)$.
If we pick $x_0=-1$, and set $y_0=-1+\sqrt{2}$, then $(x_0,y_0)$ is a $K$-rational point on $X_1(11)$ satisfying $f_{11}(x_0,y_0)=0$, where $K=\Q(\sqrt{2})$ is a quadratic number field.
If we apply the formulas in \cite[Table 7]{Su} and \eqref{eq:nP} with $x=x_0$ and $y=y_0$, we obtain
\begin{align*}
b_0:=b(x_0,y_0)=&8 - 6\sqrt{2},\\ c_0:=c(x_0,y_0)=&2 - 2\sqrt{2}.
\end{align*}
Then, the elliptic curve $E(b_0,c_0)$ over $K$ contains the point $(0,0)$ of order $11$, and in fact its torsion subgroup is $\Z/11\Z$.

From~\cite{Su}, we obtain the defining equations of $X_1(N)$ in \cite[Table 6]{Su} and birational maps $\varphi$ for $X_1(N)$ from $f_N(x,y) = 0$ to $F_N(r, s) = 0$ in \cite[Table 7]{Su} for $N\leq 30$, where $F_N(r, s) = 0$ denotes the raw form of $X_1(N)$.

Now we explain another method to solve the moduli problem of $X_1(N)$ introduced by Baaziz \cite{B}.
Let $\H$ be the complex upper half plane and $\H^*=\H\cup\P^1(\Q)$.
Then, $\Gamma_1(N)$ acts on $\H^*$ under linear fractional transformations, and $X_1(N)(\C)$ can be viewed as a Riemann surface $\Gamma_1(N)\backslash\H^*$.

The points of $\Gamma_1(N)\backslash \H$ have a one-to-one correspondence with the equivalence classes of elliptic curves $E$, together with a specified point $P$ of exact order $N$.
Let $L_\tau=[\tau,1]$ be the lattice in $\C$ with basis $\tau$ and $1$.
Then, $[\tau]\in\Gamma_1(N)\backslash \H$ corresponds to the pair $\left[\C/L_\tau,\frac{1}{N}+L_\tau\right]$.
Thus, $\Gamma_1(N)\backslash \H$ is a moduli space for the moduli problem of determining equivalence classes of pairs $(E,P)$, where $E$ is an elliptic curve defined over $\C$, and $P\in E$ is a point of exact order $N$.
Two pairs $(E,P)$ and $(E',P')$ are equivalent if there is an isomorphism $E\simeq E'$ that takes $P$ to $P'$.

Note that
\begin{align}\label{eq}\notag
&\left[\C/L_\tau, \,  \frac{1}{N}+L_\tau\right]\\
 &=  [y^2=4x^3-g_2(\tau)x-g_3(\tau),(\wp(\frac{1}{N},\tau),\wp'(\frac{1}{N},\tau))] \\ \notag
&=\left[y^2+(1-c(\tau))xy-b(\tau)y=x^3-b(\tau)x^2,\, (0,0)\right],
\end{align}
where $g_2(\tau)=60G_4(\tau)$, $g_3(\tau)=140G_6(\tau)$ for the Eisenstein series $G_{2k}(\tau)$ of weight $2k$, $\wp(z,\tau):=\wp(z, L_\tau)$ is the Weierstrass elliptic function, and $b(\tau)$, $c(\tau)$ are the coefficients of the Tate normal form contained in $\left[\C/L_\tau, \,  \frac{1}{N}+L_\tau\right]$.
Also, if we let
$$a(\tau)=\frac{1}{12}\left\{(1-c(\tau))^2-4b(\tau)\right\},$$
then we have the following relations:
\begin{align*}
g_2(\tau)=&12a(\tau)^2 + 2b(\tau)(1-c(\tau))\\
g_3(\tau)=&-8a(\tau)^3 - 2a(\tau)b(\tau)(1-c(\tau)) - b(\tau)^2,
\end{align*}
which gives the last equality of \eqref{eq} (See \cite{B} for more details.).

Note that each equivalence class of pairs $(E,P)$ contains a unique Tate normal form \cite[Proposition 1.3]{B}, and hence $b(\tau)$ and $c(\tau)$ induce well-defined functions on $\Gamma_1(N)\backslash \H$.
From \cite{B}, it follows that
\begin{align}\label{eq:b-c}
b(\tau)&=-\frac{(\wp(\frac{1}{N},\tau)-\wp(\frac{2}{N},\tau))^3}{\wp'(\frac{1}{N},\tau)^2},\\ \notag
c(\tau)&=-\frac{\wp'(\frac{2}{N},\tau)}{\wp'(\frac{1}{N},\tau)}
\end{align}
are modular functions on $\Gamma_1(N)$ and
generate the function field of $X_1(N)$, where $\wp'$ is the derivative with respect to $z$.
Thus, for each non-cuspidal point $\tau\in X_1(N)$, an elliptic curve $E(b(\tau),c(\tau))$ with $N$-torsion point $(0,0)$ corresponds.

Finally, we decribe some automorphisms of $X_\Delta(N)$.
Note that $X_\Delta(N)\to X_0(N)$ is a Galois covering with Galois group $\Gamma_0(N)/\Gamma_\Delta(N)$ which is isomorphic to $(\Z/N\Z)^*/\Delta$. 
For an integer $a$ prime to $N,$ let $[a]$ denote the automorphism of $X_\Delta(N)$ represented by $\gamma\in\Gamma_0(N)$ such that $\gamma\equiv\sm a&*\\0&*\esm\mod N.$ Sometimes we regard $[a]$ as
a matrix.

For each divisor $d|N$ with $(d,N/d)=1,$ consider the matrices of the form $\begin{pmatrix} dx & y\\Nz & dw \end{pmatrix}$ with
$x,y,z,w\in\mathbb Z$ and determinant $d.$ 
Then these matrices define a unique involution of $X_0(N)$, which is called the {\it Atkin-Lehner involution} and denoted by $W_d$. 
In particular, if $d=N,$ then $W_N$ is called the {\it Fricke involution.}
We also denote by $W_d$ a matrix of the above form.

If we fix a matrix $W_d$ then $W_d$ may not belong to the normalizer $N(\Gamma_\Delta(N))$ of $\Gamma_\Delta(N)$ in ${\rm PSL}_2(\mathbb R)$ and might therefore not define an automorphism of $X_\Delta(N)$.
However $W_d$ is always contained in $N(\Gamma_1(N))$; hence it defines an automorphism of $X_1(N)$.
But it may not define an involution of $X_1(N)$.
For example, $W_2=\begin{pmatrix} 2 & 1\\26 & 14 \end{pmatrix}$ defines an automorphism of $X_1(26)$ of order $8$.
Moreover, if $a$ is running through $(\Z/N\Z)^*/\Delta$, then $[a]W_d$ gives the different automorphisms of $X_{\Delta}(N)$ that induce the same Atkin-Lehner involution $W_d$ on $X_0(N)$.

For square-free $N$, Momose claimed that the automorphism group ${\rm Aut}(X_\Delta(N))$ is equal to $N(\Gamma_\Delta(N))/\Gamma_\Delta(N)$
in his unpublished work \cite{Mo}.
But the author, Kim and Schweizer \cite{JKS2} found a counterexample for the case $X_\Delta(37)$ where $\Delta = \{\pm1, \pm6, \pm8, \pm10, \pm11, \pm14\}$.
Therefore, it can be concluded that the automorphism group of $X_1(N)$ has not yet been completely determined even for square-free $N$.



\section{Non-tetraelliptic curves}\label{sec:non-tetraelliptic}

For determining non-tetraelliptic curves $X_1(N)$, we first give lists of rational, elliptic, bielliptic or trielliptic curves $X_1(N)$ as follows:

\begin{thm}\label{list} {\rm (\cite{JK,J})} The following holds:
\begin{enumerate}
\item[(1)] $X_1(N)$ is rational if and only if $N=1-10,12$.
\item[(2)] $X_1(N)$ is elliptic if and only if $N=11,14,15$.
\item[(3)] $X_1(N)$ is bielliptic if and only if $N=13,16,17,18,20,21,22,24$
\item[(4)] $X_1(N)$ is trielliptic over $\Q$ if and only if $N=19,21,22,27,30$
\end{enumerate}
\end{thm}

Let $X$ be a smooth, projective curve over a field $K$.
The $K$-gonality ${\rm Gon}_K(X)$ of $X$ is the minimal degree of a finite $K$-rational map $X\to \P^1$.
\begin{prop} {\rm (\cite{A}, \cite{Ki})} \label{lambda} Let $\Gamma\subseteq {\rm SL}_2(\Z)$ be a congruence subgroup.
The $\C$-gonality of the modular curve $X(\Gamma)$ is at least $\frac{\lambda_1}{24}[{\rm SL}_2(\Z) : \Gamma]$, where $\lambda_1 \geq \frac{975}{4096}$.
\end{prop}

Suppose $X$ is tetraelliptic, and there exists a $\Q$-rational tetraelliptic map $X\to E$ for some elliptic curve $E$.
Since there exists a $\Q$-rational map $E\to\P^1$ of degree 2, ${\rm Gon}_\Q(X)\leq 8$.
Since ${\rm Gon}_\C(X)\leq {\rm Gon}_\Q(X)$, applying Proposition \ref{lambda}, we have the following:

\begin{lem}\label{lem37}
$X_1(N)$ is not tetraelliptic over $\Q$ for $N=41,43,45,47$, and $N\geq 49$.
\end{lem}
\begin{proof} Note that $[{\rm SL}_2(\Z) : \Gamma]=\frac{\varphi(N)\psi(N)}{2}$, where $\psi(N)=N\displaystyle\prod_{p|N\atop{p: prime}}\left(1+\frac{1}{p}\right)$.
By Proposition \ref{lambda}, ${\rm Gon}_\C(X_1(N))\geq \frac{325\varphi(N)\psi(N)}{65536}$, and the result follows from the inequality
$$\frac{325\varphi(N)\psi(N)}{65536}>8.$$
\end{proof}

On the other hand, Derickx and van Hoeij \cite[Table 1]{DH} computed ${\rm Gon}_\Q(X_1(N))$ for $N\leq 40$ and gave upper bounds for $N\leq 250$.
The following result comes from the table.

\begin{lem}\label{lem29}
$X_1(N)$ is not tetraelliptic over $\Q$ for $N=29, 31,33-35,37-40,42,44,46,48$.
\end{lem}
The fact that ${\rm Gon}_\Q(X_1(N))>8$ is found in \cite[Table 1]{DH} for $N$ in Lemma \ref{lem29} except for $N=42,44,46,48$, and the fact that ${\rm Gon}_\Q(X_1(42))>8$ is on p. 67 in \cite{DH}.

Suppose there exists a $\Q$-rational tetraelliptic map $X_1(N)\to E$ for some elliptic curve $E$.
Then the conductor of $E$ divides $N$. By using this fact together with \cite[Table1]{C}, we have the following:

\begin{lem}\label{lem13}
$X_1(N)$ is not tetraelliptic over $\Q$ for $N=13,16,18,22,23,25$.
\end{lem}

We recall a basic fact about modular parametrization.

\begin{prop}\cite[(1.4) Proposition]{S}\label{optimal} Let $\mathcal E$ be an isogeny class (over $\Q$) of elliptic curves of conductor $N$.
Then there exists an elliptic curve $E_1$ and a modular parametrization
$$\pi_1:X_1(N)\to E_1$$
such that if $\pi:X_1(N)\to E$ is a parametrization of an elliptic curve $E\in\mathcal E$, 
then there is an isogeny $\phi: E_1\to E$ which makes the following diagram commutative:

$$\begin{tikzcd}
X_1(N) \arrow[rd, "{\pi'}"] \arrow[r, "\pi"] & E_1 \arrow[d, "\phi"] \\
& E
\end{tikzcd}$$

\end{prop}

The elliptic curve $E_1$ of Proposition \ref{optimal} is called the {\it $X_1(N)$-optimal curve}.
On the other hand, there exists an elliptic curve $E_0$ having the same property described in Proposition \ref{optimal} for the modular parametrizaions of $X_0(N)$, and it is called the {\it $X_0(N)$-optimal curve}.

\begin{lem}
$X_1(19)$ and $X_1(27)$ are not tetraelliptic over $\Q$.
\end{lem}
\begin{proof} Suppose that there is a $\Q$-rational tetraelliptic map $\phi:X_1(19) \to E$ for some elliptic curve $E$.
Then $E$ is of conductor $19$. 
By Theorem \ref{list}, $X_1(19)$ is trielliptic over $\Q$, and so there exists a $\Q$-rational map $\phi':X_1(19)\to E'$ of degree $3$ for some elliptic curve $E'$ of conductor $19$.
According to \cite[Table 1]{C}, there exists only one isogeny class  of conductor 19, and hence $E$ and $E'$ are $\Q$-isogenous.
Since $X_1(19)$ is neither rational nor elliptic nor bielliptic, $E'$ is $X_1(19)$-optimal but $\phi$ cannot factor through $\phi'$ because of their degrees; hence we have a contradiction.
Thus $X_1(19)$ is not tetraelliptic over $\Q$.

By the exact same proof, one can prove that $X_1(27)$ is not tetraelliptic over $\Q$.
\end{proof}

\begin{lem}\label{lem:36}
$X_1(36)$ is not tetraelliptic over $\Q$.
\end{lem}
\begin{proof} There is only one isogeny class of elliptic curves over $\Q$ with conductor dividing $36$, and this class has conductor $36$. 
Actually, it is the isogeny class of $X_0(36)$.
Now suppose that $X_1(36)$ is tetraelliptic over $\Q$. 
By Proposition \ref{optimal}, there is an elliptic curve $E$ in this isogeny class which is optimal for the uniformization from $X_1(36)$. 
The degree from $X_1(36)$ to $E$ would have to divide $4$ (from the tetraelliptic map) and it would have to divide $6$ (from the degree $6$ map from $X_1(36)$ to the elliptic curve $X_0(36)$). 
So $X_1(36)$ would be bielliptic, which is a contradiction to Theorem \ref{list}.
\end{proof}

\begin{lem}
$X_1(21)$ is not tetraelliptic over $\Q$.
\end{lem}
\begin{proof} By following the exact same proof of Lemma \ref{lem:36}, one can calculate that if $X_1(21)$ is tetraelliptic over $\Q$ then it must be bielliptic over $\Q$.
However $X_1(21)$ does not allow a $\Q$-rational map to an elliptic curve of degree $2$ by \cite[Lemma 3.5]{JK}.
\end{proof}

\begin{rmk} $X_1(21)$ is tetraelliptic but not over $\Q$.
Indeed, there exist a map $X_1(21)\to X_{\Delta}(21)$ of degree 2, where $\Delta=\{\pm1,\pm8\}$.
By \cite[Table1]{JKS2}, we know that $X_{\Delta}(21)$ is a biellipitic curve with a bielliptic involution $W_{21}$; hence $X_1(21)$ is a tetraelliptic curve. We note that $W_{21}$ is not defined over $\Q$.
\end{rmk}

\section{Tetraelliptic curves over $\Q$}\label{sec:tetraelliptic}


For the reader's convenience, we restate Theorem \ref{main} as follows: $X_1(N)$ is tetraelliptic over $\Q$ if and only if $N=17,20,24,26,28,30,32$.
Moreover they admit $\Q$-rational Galois maps, say $\phi_N$, from $X_1(N)$ to elliptic curves of degree 4.
Especially for $N=17,20,26,30$ and $32$, $X_1(N)$ is cyclic tetraelliptic, and the others are not. 

In this section, for $N$ listed above, we show $X_1(N)$ are tetraelliptic over $\Q$, and construct tetraelliptic maps $\phi_N:X_1(N)\to E_N$ for some elliptic curves $E_N$.
Moreover, we find elliptic curves with torsion subgroup $\Z/17\Z$ over cyclic quartic number number fields by using $\phi_{17}$. 


\subsection{Proof of Theorem \ref{main}}

For $N=17,20,24,28,32$, there are natural tetraelliptic maps from $X_1(N)$ to elliptic curves,
namely $\phi_{17}:X_1(17)\to X_{\Delta_1}(17)$, $\phi_{20}:X_1(20)\to X_0(20)$, $\phi_{24}:X_1(24)\to X_0(24)$, $\phi_{28}:X_1(28)\to X_1(14)$, and $\phi_{32}:X_1(32)\to X_{\Delta_2}(32)$ where $\Delta_1=\{\pm1,\pm2,\pm4,\pm8\}$ and $\Delta_2=\{\pm1,\pm7,\pm9,\pm15\}$. 
Then $G_{17}=\langle[2]\rangle$, $G_{20}=\langle[3]\rangle$ and $G_{32}=\langle[7]\rangle$ are cyclic groups of order $4$, and $G_{24}=\langle[5],[7]\rangle$ and $G_{28}=\langle[13],\nu\rangle$ are Klein $4$-groups, where $\nu=\begin{pmatrix}1&0\\14&1\end{pmatrix}$.
We will determine later whether each target curve of $\phi_N$ is the same as any of elliptic curves labeled in \cite{Cr}.

Now consider the remaining cases, i.e. $X_1(26)$ and $X_1(30)$.
Let $\mathcal E$ denote an isogeny class of elliptic curves defined over $\Q$ of conductor $N$. 
In general, the $X_0(N)$-optimal curve $E_0\in\mathcal E$ and the $X_1(N)$-optimal curve $E_1\in\mathcal E$ are different.
For example, $E_0 = X_0(11)$ and $E_1 = X_1(11)$ differ by a $5$-isogeny. 
Stein and Watkins \cite{SW} have made a precise conjecture about when $E_0$ and $E_1$ differ by a $3$-isogeny, based on
numerical observation. 
Lee \cite{Y} proved some partial results about this conjecture and gave an example of an isogeny class of conductor $26$ in which the $X_0(26)$-optimal curve 26A1 and $X_1(26)$-optimal curve 26A3 differ by a $3$-isogeny.
According to Table 5 in \cite{Cr}, the modular parametrization $X_0(26)\to {\rm 26A1}$ is of degree $2$ and the natural map $X_1(26)\to X_0(26)$ is of degree $6$; hence their composition $X_1(26)\to {\rm 26A1}$ is of degree $12$.
By Proposition \ref{optimal}, this map factors through the modular parametrization $X_1(26)\to {\rm 26A3}$, and 26A3 and 26A1 differ by $3$-isogeny; hence $X_1(26)\to {\rm 26A3}$ must be of degree $4$.
Thus $X_1(26)$ is tetraelliptic over $\Q$.

Now we find out what the tetraelliptic map $\phi_{26}:X_1(26)\to {\rm 26A3}$ is.
Bars \cite{B} proved that $X_0(26)$ is a bielliptic curve, and $W_2$ and $W_{13}$ are all of its bielliptic involutions.
According to LMFDB \cite{LM}, the corresponding cusp form $f$ of weight $2$ of 26A1 has $-1$ as its second Fourier coefficient, and so the eigenvalue of $f$ by $W_2$ is $+1$; hence the quotient space $X_0(26)/W_2$ is equal to 26A1.
For your information $X_0(26)/W_{13}$ is equal to 26B1.
Then 26A1 should be $X_1(26)/G$ where $G$ is the group of order $12$ generated by $Gal(X_1(26)/X_0(26))=\langle [7]\rangle$ and $W_2$.
Since $W_2$ as an automorphism of $X_1(26)$ is of order $12$, the group $G=\langle W_2\rangle$ is cyclic. 
Now $\phi_{26}$ is a Galois map with a subgroup of $G$ of order 4 as the Galois group.
Since $G$ has only one subgroup $\langle W_2^3\rangle$ of order $4$, the tetraelliptic map $\phi_{26}$ should be a cyclic map with Galois group $\langle W_2^3\rangle$.

Next we consider the case of $X_1(30)$.
Assume there is a $\Q$-rational tetraelliptic map $\phi:X_1(30) \to E$ for some elliptic curve $E$. 
Then $E$ has conductor $15$ or $30$.
If $E$ is of conductor $30$, then $E$ is $X_1(30)$-optimal by Theorem \ref{list}.
The $X_0(30)$-optimal curve in the isogeny class of $E$ is 30A1. 
By \cite{B}, $X_0(30)$ has exactly $3$ bielliptic involutions, namely $W_5, W_6$ and $W_{30}$.
By the same method as in $X_1(26)$, one can check that $X_1(30)/W_5$ is equal to 30A1.
Let $F$ be the compositum of the function fields of $X_0(30)$ and $E$, which is contained in the function field of $X_1(30)$. 
But since $X_1(30)$ is cyclic of degree 4 over $X_0(30)$, there is only one intermediate field. 
So $F$ must be the function field of $X_{\Delta}(30)$ where $\Delta=\{\pm1,\pm11\}$. 
This means that the map $\phi$ factors through $X_{\Delta}(30)$.
So $\phi$ must be given by a subgroup of order $4$ of the order $8$ group generated by $Gal(X_1(30)/X_0(30))=\langle [7]\rangle$ and $W_5$.
Since $G$ is a dihedral group, it has $3$ subgroups of order $4$.
By using the method that we will explain later, we compute the quotient space of $X_1(30)$ by each subgroup of order $4$ but there was no case in which the quotient space is an ellptic curve over $\Q$, which is a contradiction.

Thus $E$ should be of conductor 15.
However, it is tricky to deal with this case, and there are not many known results.
So we tried to calculate the quotient spaces by subgroups of order $4$ of ${\rm Aut}(X_1(30))$, and eventually got the desired result.
Consider the subgroup of order 8 generated by $[7]$ and $W_6$, which is a dihedral group, and let $H=\langle [11],[7]W_6\rangle$ be a Klein 4-subgroup.
Then the quotient space $X_1(30)/H$ is equal to 15A8, and so $\phi_{30}:X_1(30)\to{\rm 15A8}$ is a $\Q$-rational tetraelliptic map, and $X_1(30)$ is tetraelliptic over $\Q$.

Summarizing all of the above results, we finish the proof of Theorem \ref{main}.

\subsection{Construction of tetraelliptic maps}
From now on, we construct tetraelliptic maps $\phi_{N}:X_1(N)\to E_N$ explicitly, i.e. we find the rational expressions of $\phi_N$ in terms of $x$ and $y$ satisfying defining equations $f_N(x,y)=0$ of $X_1(N)$ as introduced in Section \ref{sec:Preliminaries}.

First consider the case of $X_1(17)$.
Then the map $X_1(17)\to X_{\Delta_1}(17)=X_1(17)/\langle [2]\rangle$ is a tetraelliptic map as explained above.
If we take $[2]=\sm 2&1\\17&9\esm$, then $[2]$ acts on $X_1(17)$ as $[2]\tau=\frac{2\tau+1}{17\tau+9}$.
In this case, we have the following:
\begin{align*}
\wp \left(\frac{1}{17}, [2]\tau\right)  &= (17\tau+9)^2 \wp\left(\frac{17\tau+9}{17}, \tau\right) \\
&= (17\tau+9)^2 \wp\left(\frac{9}{17}, \tau\right),
\end{align*}
and
\begin{align*}
\wp' \left(\frac{1}{17}, [2]\tau\right)  &= (17\tau+9)^3 \wp'\left(\frac{17\tau+9}{17}, \tau\right) \\
&= (17\tau+9)^3 \wp'\left(\frac{9}{17}, \tau\right).
\end{align*}
Thus, from \eqref{eq:b-c}, we obtain the following:
\begin{align}\label{eq:b-c2}
b([2]\tau) &= -\dfrac{ \left( \wp \left(\frac{9}{17}, \tau\right) - \wp \left(\frac{1}{17}, \tau\right) \right)^3}{\wp' \left(\frac{9}{17}, \tau\right)^2 },\\
c([2]\tau) &= -\dfrac{  \wp' \left(\frac{1}{17}, \tau\right) }{\wp' \left(\frac{9}{17}, \tau\right) }.\nonumber
\end{align}
Similarly, we have
\begin{align}\label{eq:b-c4}
b([4]\tau) &= -\dfrac{ \left( \wp \left(-\frac{4}{17}, \tau\right) - \wp \left(-\frac{8}{17}, \tau\right) \right)^3}{\wp' \left(-\frac{4}{17}, \tau\right)^2 },\\
c([4]\tau) &= -\dfrac{  \wp' \left(-\frac{8}{17}, \tau\right) }{\wp' \left(-\frac{4}{17}, \tau\right) },\nonumber
\end{align}
and
\begin{align}\label{eq:b-c8}
b([8]\tau) &= -\dfrac{ \left( \wp \left(-\frac{2}{17}, \tau\right) - \wp \left(-\frac{4}{17}, \tau\right) \right)^3}{\wp' \left(-\frac{2}{17}, \tau\right)^2 },\\
c([8]\tau) &= -\dfrac{  \wp' \left(-\frac{4}{17}, \tau\right) }{\wp' \left(-\frac{2}{17}, \tau\right) }.\nonumber
\end{align}

From \eqref{eq:nP} and \cite[Table 7]{Su}, we have that the generators $x,y$ of the function field of $X_1(17)$ satisfying $f_{17}(x,y)=0$ can be expressed as the following functions of $b,c$:
\begin{align}\label{eq:x-y}
x&=-\frac{(-c^2 + b - c)(-c^3 + b^2 - bc)}{bc^4 + c^5 - 3b^2c^2 + 3bc^3 + b^3 - 2b^2c + bc^2},\\
y&=-\frac{(-c^2 + b - c)(-bc^2 + 2b^2 - 3bc + c^2)}{bc^4 + c^5 - 3b^2c^2 + 3bc^3 + b^3 - 2b^2c + bc^2}.\nonumber
\end{align}

From the formulas in Proposition 3 of \cite[p. 46]{L}, we can determine the $q$-expansions for $\wp(z,\tau)$ and $\wp'(z,\tau)$, where $q=e^{2\pi i\tau}$. Using these $q$-expansions and \eqref{eq:b-c}, \eqref{eq:b-c2}, and \eqref{eq:x-y}, we arrive at the following $q$-expansions:
\begin{align*}
x(\tau)=&-\o^{12} - \o^{11} - \o^{10} - 2\o^9 - 2\o^8 - \o^7 - \o^6 - \o^5+O(q),\\
y(\tau)=& \o^{15} - 2\o^{14} - 2\o^{13} + \o^{12} - \o^{11} - 3\o^{10} - 3\o^7 - \o^6 \\
&+ \o^5 - 2\o^4 - 2\o^3 + \o^2 - 3+O(q),\\
x([2]\tau)=&-\o^{14} - 2\o^{13} - \o^{12} - \o^{11} - \o^6 - \o^5 - 2\o^4 - \o^3+O(q),\\
y([2]\tau)=&-3\o^{15} - 2\o^{14} - \o^{13} - 4\o^{12} - 3\o^{10} - \o^9 - \o^8 - 3\o^7\\
& - 4\o^5 - \o^4 - 2\o^3 - 3\o^2 - 4+O(q),\\
x([4]\tau)=&-2\o^{15} - \o^{14} - \o^{11} - \o^{10} - \o^7 - \o^6 - \o^3 - 2\o^2+O(q),\\
y([4]\tau)=&2\o^{15} + 3\o^{14} + 2\o^{13} - \o^{11} + \o^{10} + 3\o^9 + 3\o^8 + \o^7 \\
&- \o^6 + 2\o^4 + 3\o^3 + 2\o^2 - 1+O(q),\\
x([8]\tau)=&2\o^{15} + \o^{14} + 2\o^{13} + \o^{12} + 2\o^{11} + \o^{10} + 2\o^9 + 2\o^8 \\
&+ \o^7 + 2\o^6 + \o^5 + 2\o^4 + \o^3 + 2\o^2 + 2+O(q),\\
y([8]\tau)=&-3\o^{14} + \o^{13} - \o^{12} - 2\o^{11} + \o^{10} - 2\o^9 - 2\o^8 + \o^7 \\
&- 2\o^6 - \o^5 + \o^4 - 3\o^3 - 3+O(q),
\end{align*}
where $\o$ is a $21$th primitive root of $1$.

Now we let $\bar{x}(\tau)=x(\tau)+x([2]\tau)+x([4]\tau)+x([8]\tau)$ and $\bar{y}(\tau)=y(\tau)+y([2]\tau)+y([4]\tau)+y([8]\tau)$, and then $\bar{x},\bar{y}$ generate the function field of $X_{\Delta_1}(17)$.
Using them we can obtain a defining equation of  $X_{\Delta_1}(17)$ by the following the algorithm.

\begin{alg}\label{alg1}
\begin{enumerate}
\item[(1)] Set a polynomial $f(x, y)=\displaystyle\sum_{i+j\leq d} a_nx^i y^j$ for a positive integer $d$ and unknown $a_n$.
\item[(2)] Input the $q$-expansions of $\bar{x}(\tau)$ and $\bar{y}(\tau)$ in $f(x,y)$.
\item[(3)] For the equation $g(\bar{x},\bar{y})=0$, compare the coefficients of $q^m$ for $0\leq m\leq \frac{(d+1)(d+2)}{2}$, and set a system of linear equations.
\item[(4)] Solve the system of (3), and obtain a defining equation.
\end{enumerate}
\end{alg}

By using Algorithm \ref{alg1}, we obtain a defining equation of $X_{\Delta_1}(17)$ as follows:
\begin{align*}
f(x,y)=&-x^3y^2 + x^4 - 4x^3y - 4x^2y^2 - 4x^3 - 16x^2y - 10xy^2\\
& - y^3 - 14x^2 - 46xy - 17y^2 - 60x - 70y - 100=0.
\end{align*}

By using the computer algebra system Maple, we can convert $g(x,y)$ into a Weierstrass equation, and we have the following:
$$E_{17}: v^2+uv+v = u^3 - u^2 - 6u -4.$$
Thus $E_{17}$ is equal to 17A2 whose torsion group is $\Z/4\Z$. Here an isomorphism between $X_{\Delta_1}(17)$ and $E_{17}$ is given by
\begin{align}\label{uv}
u(x,y)=&(x^5 + x^4y + 7x^4 + 5x^3y + 26x^3 + 15x^2y + xy^2 + 63x^2 + 27xy + y^2 \\ \nonumber
&+ 90x + 17y + 50)/(x^4 + 4x^3 + 12x^2 + 18x + 10)\\ \nonumber
v(x,y)=&(x^5y + 3x^5 + 7x^4y + 16x^4 + 25x^3y + x^2y^2 + 56x^3 + 58x^2y + 3xy^2 \\ \nonumber
&+ 121x^2 + 73xy + 3y^2 + 136x + 41y + 70)/(x^4 + 4x^3 + 12x^2 + 18x + 10)
\end{align}

\begin{rmk}
By the exact same method as above, we could find the target elliptic curves $E_N$ for $N=20,24,28,32$ are 20A1,24A1,14A4,32A2, respectively.
\end{rmk}

We compute the $q$-expansions of two generators $u(\tau)$ and $v(\tau)$ of the function field of $E_{17}$ by using \eqref{uv} and the $q$-expansions of $\bar{x}(\tau)$ and $\bar{y}(\tau)$ as follows:
\begin{align*}
u(\tau)=&-2\o^{14} - 2\o^{12} - 2\o^{11} - 2\o^{10} - 2\o^7 - 2\o^6 - 2\o^5 - 2\o^3 + 2+O(q),\\
v(\tau)=& -3\o^{14} - 3\o^{12} - 3\o^{11} - 3\o^{10} - 3\o^7 - 3\o^6 - 3\o^5 - 3\o^3 + 5+O(q).
\end{align*}


Now we obtain the explicit expression of $\phi_{17}:X_1(17)\to E_{17}$ by the following algorithm.

\begin{alg}\label{alg2}
\begin{enumerate}
\item[(1)] Set two polynomials $g(x, y)=\displaystyle\sum_{i+j\leq d} a_nx^i y^j$ and $h(x, y)=\displaystyle\sum_{i+j\leq d} b_n x^i y^j$ for a positive integer $d$ and unknown $a_n$ and $b_n$.
\item[(2)] Input the $q$-expansions of $x(\tau)$ and $y(\tau)$ in $g(x,y)$ and $h(x,y)$.
\item[(3)] For each of two equations $g(x(\tau),y(\tau))=u(\tau)h(x(\tau),y(\tau))$ and $g(x(\tau),y(\tau))=v(\tau)h(x(\tau),y(\tau))$, compare the coefficients of $q^m$ for $0\leq m\leq (d+1)(d+2)$, and set two system of linear equations.
\item[(4)] Solve the systems of (3), we obtain a rational map $\left(\frac{g_1(x,y)}{h_1(x,y)},\frac{g_2(x,y)}{h_2(x,y)}\right)$, where
$g_1(x,y)$ and $h_1(x,y)$ (resp. $g_2(x,y)$ and $h_2(x,y)$) are obtain by solving the system of linear equations from $g(x(\tau),y(\tau))=u(\tau)h(x(\tau),y(\tau))$ (resp. $g(x(\tau),y(\tau))=v(\tau)h(x(\tau),y(\tau))$).
\end{enumerate}
\end{alg}

Finally, by using Algorithm \ref{alg2}, we obtain an explicit expression of the tetraelliptic map $\phi_{17}:X_1(17)\to E_{17}$, say $\phi_{17}(x,y)=\left(p_1(x,y),p_2(x,y)\right)$, as follows:

\begin{align*}
p_1(x,y)=&\frac{x^2 + xy - y^2 + x - 2y - 1}{(y + 1)^2},\\
p_2(x,y)=&\frac{x^3y^2 - x^4 + 2x^3y + 2x^2y^2 - x^3 + x^2y + y^3 - xy + y^2}{(y + 1)^2(x - y)}.
\end{align*}

Next we will construct the tetraelliptic map $\phi_{26}:X_1(26)\to X_1(26)/\langle W_2^3\rangle={\rm 26A3}.$
If we take $W_2=\begin{pmatrix}2&1\\26&14\end{pmatrix}$, then $W_2$ acts on $X_1(26)$ as $W_2\tau=\frac{2\tau+1}{26\tau+14}$.
In this case, we have the following:
\begin{align*}
\wp \left(\frac{1}{26}, W_2\tau\right)  &= (26\tau+14)^2 \wp\left(\frac{26\tau+14}{26}, 2\tau\right) \\
&= (26\tau+14)^2 \wp\left(\tau+\frac{14}{26}, 2\tau\right),
\end{align*}
and
\begin{align*}
\wp' \left(\frac{1}{26}, W_2\tau\right)  &= (26\tau+14)^3 \wp'\left(\frac{26\tau+14}{26}, 2\tau\right) \\
&= (26\tau+14)^3 \wp'\left(\tau+\frac{14}{26}, \tau\right).
\end{align*}
Thus, from \eqref{eq:b-c}, we obtain the following:
\begin{align}\label{eq:b-cw2}
b(W_2\tau) &= -\dfrac{ \left( \wp \left(\tau+\frac{14}{26}, 2\tau\right) - \wp \left(\frac{2}{26}, 2\tau\right) \right)^3}{\wp' \left(\tau+\frac{14}{26}, 2\tau\right)^2 },\\
c(W_2\tau) &= -\dfrac{  \wp' \left(\frac{2}{26}, 2\tau\right) }{\wp' \left(\tau+\frac{2}{26}, 2\tau\right) }.\nonumber
\end{align}

By using this action and the exact same method as in the case of $X_1(17)$, we can find an explicit expression of $\phi_{26}$ as in Table \ref{table}.
We can treat the other cases by the exact same procedure.
We list in Table \ref{table} explicit expressions of $\phi_N$ for the list of $N$ in Theorem \ref{main}.

\begin{center}

\begin{longtable}{l|l}
\caption{Explicit expressions of $\phi_N$}\label{table}\\

\hline $N$ & Explicit expressions of $\phi_N=(p_1(x,y),p_2(x,y))$ \\ \hline
$17$ &  $\begin{array}{l}
p_1(x,y)=\frac{x^2 + xy - y^2 + x - 2y - 1}{(y + 1)^2},\\
p_2(x,y)=\frac{x^3y^2 - x^4 + 2x^3y + 2x^2y^2 - x^3 + x^2y + y^3 - xy + y^2}{(y + 1)^2(x - y)}
\end{array}$   \\ \hline
$20$ &  $\begin{array}{l}
p_1(x,y)=\frac{x^3y - 2x^2y + 2xy + y^2 + 2y + 2}{(y + 1)(x - 1)},\\
p_2(x,y)=-\frac{2x^3y^2 - xy^4 - 2x^2y^2 + 2xy^3 + 3y^4 - 2x^2y + 4xy^2 + 7y^3 + 4xy + 8y^2 + 2x + 4y}{(y + 1)^2(x + y)}
\end{array}$  \\ \hline
$24$ &  $\begin{array}{l}
p_1(x,y)=\frac{2x(-xy^2 + x^2 + xy - 2y^2 + 2y)}{-x^2y^2 + x^3 + x^2y - 2xy^2 + 2xy + 2y^2},\\
p_2(x,y)=\frac{4(x + 1)(-x^2y^2 + x^3 + x^2y - xy^2 + y^3 + xy)}{(x + y)(-x^2y^2 + x^3 + x^2y - 2xy^2 + 2xy + 2y^2)}
\end{array}$  \\ \hline
$26$ &  $\begin{array}{l}
p_1(x,y)=\frac{(y - 1)(x^2 + x + y)(x + y)}{x^4 + x^3y + 2x^3 + 2x^2y + xy^2 + x^2 - xy - y^2 + x + y},\\
p_2(x,y)=-\frac{(y - 1)(2xy + y^2 - x - y)}{x^4 + x^3y + 2x^3 + 2x^2y + xy^2 + x^2 - xy - y^2 + x + y}
\end{array}$  \\ \hline
$28$ &  $\begin{array}{l}
p_1(x,y)=-\frac{x^2(y - 1)(x + y)(x + 1 - y)}{x^4y + 4x^3y - 2x^3 + 5x^2y + 2xy^2 + y^3 - 3x^2 + xy - x},\\
p_2(x,y)=-\frac{y(x + 1)(-y^3 + x^2 - xy + x)}{x^2(xy^2 - xy + 2y^2 + x - 2y + 1)}
\end{array}$  \\ \hline
$30$ &  $\begin{array}{ll}
p_1(x,y)=&-\frac{y(x^3y + x^2y + x^2 - y^2 + x + y)}{x^3y^3 + 2x^2y^3 - x^2y^2 - y^4 + x^2y + 2xy^2 + 2y^3 - 2xy - 2y^2 + x + y},\\
p_2(x,y)=&{\scriptstyle -y(2x^4y^2 - x^4y + 4x^3y^2 + 2x^2y^2 - 2xy^3 - x^3 + 3x^2y + 4xy^2 - y^3 - 2x^2 - xy + y^2)/}\\
&{\scriptstyle (x^5y^2 + x^4y^3 + x^3y^4 + 3x^4y^2 + 2x^3y^3 + 2x^2y^4 + 2x^4y + 2x^3y^2 - 2x^2y^3 - xy^4 - y^5}\\
&{\scriptstyle+ 6x^3y + 6x^2y^2 + 3xy^3 + 2y^4 + x^3 - 4xy^2 - 3y^3 + 3x^2 + 4xy + 2y^2 - x - y)}
\end{array}$  \\ \hline
$32$ &  $\begin{array}{ll}
p_1(x,y)=&{\scriptstyle -(x - 1)(x + 1 - y)(x^2y^2 - xy^3 - 2x^2y + 2xy^2 - y + 1)/(x^4y^2 - 2x^3y^3 + x^2y^4- 2x^3y^2 }\\
&{\scriptstyle + 3x^2y^3 - xy^4 + 2x^3y - x^2y^2 - xy^3 - 3x^2y + 3xy^2 + x^2 - xy + y^2 - 2y + 1),}\\
p_2(x,y)=&{\scriptstyle 2(x^7y^2 - 4x^6y^3 + 6x^5y^4 - 4x^4y^5 + x^3y^6 + x^7y - 8x^5y^3 + 12x^4y^4 - 5x^3y^5 - x^6y + 6x^5y^2}\\
&{\scriptstyle - 8x^4y^3 + 4x^2y^5 - xy^6 + 3x^4y^2 - 2x^3y^3 - 5x^2y^4 + 4xy^5 - x^5 + x^4y - x^3y^2 + 13x^2y^3}\\
&{\scriptstyle - 13xy^4 + y^5 - x^4 + 4x^3y - 16x^2y^2 + 18xy^3 - 4y^4 - 2x^3 + 7x^2y - 12xy^2 + 8y^3 - x^2}\\
&{\scriptstyle + 4xy - 8y^2 + 3y)/(3x^7y^2 - 8x^6y^3 + 5x^5y^4 + 2x^4y^5 - 2x^3y^6 - 2x^7y + 6x^6y^2 - 8x^5y^3}\\
&{\scriptstyle + 8x^4y^4 - 6x^3y^5 + 2x^2y^6 - 2x^6y + 11x^5y^2 - 12x^4y^3 - x^3y^4 + 4x^2y^5 - 5x^5y + 5x^4y^2}\\
&{\scriptstyle + 3x^3y^3 - x^2y^4 - 2xy^5 + x^5 - 3x^4y + 5x^3y^2 + x^2y^3 - 4xy^4 + 2x^4 - 7x^3y - 5x^2y^2}\\
&{\scriptstyle + 10xy^3 + 3x^3 - x^2y - 3xy^2 + 2y^3 + 2x^2 - 2xy - 4y^2 + x + 2y)}
\end{array}$  \\ \hline

\end{longtable}
\end{center}

\subsection{Construction of elliptic curves}

In this subsection, we will construct the elliptic curves corresponding to the quartic points of $X_1(17)$ lying above the $\Q$-rational torsion points of $E_{17}$.

Note that $\phi_{17}:X_1(17)\to E_{17}$ is cyclic of degree $4$ and $E_{17}$ has torsion subgroup $\Z/4\Z$.
Let $P_1:=(3,-2)$ and $P_2:=(-5/4,1/8)$ which are two non-zero points of the four torsion points of $E_{17}$.
For $P_1$, by solving the system of equations 
$$p_1(x,y)=3,\, p_2(x,y)=-2,\, f_{17}(x,y)=0,$$
we obtain the four points of $X_1(17)$ lying above $P_1$ which are 
\begin{equation}\label{p1}
\left(\alpha_i,\frac{\alpha_i^2}{2}-\frac{\alpha_i}{2}-2\right)
\end{equation}
where $\alpha_i$'s are four roots of $2x^4 - 5x^3 - 7x^2 + 10x + 8$ for $i=1,2,3,4$.
The quartic number fields $\Q(\alpha_i)$ are cyclic over $\Q$ for all $i$.
Similarly, we obtain the four points lying above $P_2$ which are 
\begin{equation}\label{p2}
(\beta_i,-2\beta_i-1)
\end{equation}
where $\beta_i$'s are four roots of $8x^4 + 14x^3 - 11x^2 - 11x - 2$ for $i=1,2,3,4$.
The quartic number fields $\Q(\beta_i)$ are cyclic over $\Q$ for all $i$ either.
We note that $(-1,0)$ is the point of $X_1(17)$ lying above the third non-zero point of $E_{17}$, and it is a cusp of $X_1(17)$; hence we cannot obtain the corresponding elliptic curve.

From \eqref{eq:nP} and Table 7 of \cite{Su}, the elliptic curves corresponding to \eqref{p1} are given by
\begin{equation}\label{e1}
E_{1i}:=E\left(- \frac{1337}{32}\alpha_i^3 + \frac{3233}{64}\alpha_i^2 + \frac{12285}{64}\alpha_i + \frac{1135}{16},- \frac{59}{16}\alpha_i^3 + \frac{147}{32}\alpha_i^2 + \frac{535}{32}\alpha_i + \frac{45}{8}\right),
\end{equation}
and the elliptic curves corresponding to \eqref{p1} are given by
\begin{equation}\label{e2}
E_{2i}:=E\left(-\frac{6901}{1024}\beta_i^3 - \frac{41387}{4096}\beta_i^2 + \frac{91011}{8192}\beta_i + \frac{39449}{8192},-\frac{139}{16}\beta_i^3 - \frac{725}{64}\beta_i^2 + \frac{2301}{128}\beta_i + \frac{871}{128}\right).
\end{equation}
Then $E_{1i}$(resp. $E_{2i}$) are elliptic curves over cyclic quartic extensions $\Q(\alpha_i)$(resp. $\Q(\beta_i)$) over $\Q$ with torsion subgroup $\Z/17\Z$ for all $i$.

We have tried the same method in all other cases, but unfortunately we have not achieved the same results as $17$.
In all other cases, the points on $X_1(N)$ lying above the torsion points are cusps, so the corresponding elliptic curves could not be obtained.

\section*{Acknowledgements}
We are grateful to  Andreas Schweizer for discussing various results in this paper.


\end{document}